\documentclass[10pt, draft]{amsart}
\usepackage{amssymb}
\usepackage{amsthm}
\usepackage{amsmath}
\usepackage{amsfonts}
\usepackage[mathscr]{eucal}
\usepackage{enumerate}

\newtheorem{thm}{Theorem}[section]
\newtheorem{cor}[thm]{Corollary}
\newtheorem{lem}[thm]{Lemma}

\theoremstyle{definition}
\newtheorem{defn}[thm]{Definition}
\theoremstyle{remark}
\newtheorem{rem}[thm]{Remark}
\numberwithin{equation}{section}

\begin{document}
\title[Hyers--Ulam stability of derivations and linear functions]
{Hyers--Ulam stability of derivations \\ and linear functions}
\author[Z.~Boros]{Zolt\'{a}n Boros}
\address{
Department of Analysis\\
Institute of Mathematics\\
University of Debrecen\\
P. O. Box: 12.\\
Debrecen\\
H--4010\\
Hungary}

\email{boros@math.klte.hu}

\author[E.~Gselmann]{Eszter Gselmann}
\address{
Department of Analysis\\
Institute of Mathematics\\
University of Debrecen\\
P. O. Box: 12.\\
Debrecen\\
H--4010\\
Hungary}

\email{gselmann@math.klte.hu}

\begin{flushright}
{\tt Manuscript\\
\today}
\end{flushright}

\begin{abstract}

\end{abstract}

\thanks{This research has been supported by the Hungarian Scientific Research Fund
(OTKA) grants NK 68040 and K 62316)}
\maketitle

\section{Introduction and preliminaries}

In this paper $\mathbb{N}, \mathbb{Z}, \mathbb{Q}$ and
$\mathbb{R}$ denotes the set of the
natural (positive integer), the integer, the rational
and the real numbers, respectively.

The stability theory of functional equations
basically deals with the following question:
Is is true that an 'approximate' solution of a functional
equation 'can be approximated' by a solution
of the functional equation in question?
This problem was raised by S.~M.~Ulam (see \cite{Ula64})
and answered (affirmatively) by D.~H.~Hyers
concerning the additive Cauchy equation see \cite{Hye41}.
Since 1941 this result has been extended
and generalized in several ways,
see e.g., Hyers--Isac--Rassias \cite{HIR98}
and the references therein.
Of course, the question of stability
can be raised not only concerning the Cauchy equation
but also in connection with other equations.

The aim of this paper is to examine
the stability of a system of equations
that defines derivations as well as linear functions.

\begin{defn}\label{D1.1}
A function $f:\mathbb{R}\rightarrow\mathbb{R}$ is called an \emph{additive} function
if,
\begin{equation}\label{Eq1.1}
f(x+y)=f(x)+f(y)
\end{equation}
holds for all $x, y\in\mathbb{R}$. Furthermore, we say that an additive
function $f:\mathbb{R}\rightarrow\mathbb{R}$ is a \emph{derivation} if
\begin{equation}\label{Eq1.2}
f(xy)=xf(y)+yf(x)
\end{equation}
is fulfilled for all $x, y\in\mathbb{R}$.
\end{defn}
>From \eqref{Eq1.2} $f(1)=0$ follows, whence every derivation vanishes at the rationals.
Furthermore, it is known that there exist not identically zero derivations, see
Kuczma \cite{Kuc85}.

It is easy to see from the above definition that every derivation
$f:\mathbb{R}\rightarrow\mathbb{R}$ satisfies the equation
\begin{equation}\label{Eq1.3}
f(x^{k})=kx^{k-1}f(x)
\quad
\left(x\in\mathbb{R}\setminus\left\{0\right\}\right)
\end{equation}
for arbitrarily fixed $k\in\mathbb{Z}\setminus\left\{0\right\}$.
Furthermore, the converse is also true, in the following sense:
if $k\in\mathbb{Z}\setminus\left\{0, 1\right\}$ is fixed and
an additive function $f:\mathbb{R}\rightarrow\mathbb{R}$ satisfies
\eqref{Eq1.3}, then $f$ is a derivation, see e.g., Kurepa \cite{Kur64} and
Kannappan--Kurepa \cite{KK70}.

Motivated by a problem of I.~Halperin (1963), Jurkat \cite{Jur65} and
independently, Kurepa \cite{Kur64} proved that every additive function
$f:\mathbb{R}\rightarrow\mathbb{R}$ satisfying
\[
f\left(\frac{1}{x}\right)=\frac{1}{x^{2}}f(x)
\quad
\left(x\in\mathbb{R}\setminus\left\{0\right\}\right)
\]
has to be linear.

In \cite{NH68} A.~Nishiyama and S.~Horinouchi investigated
additive functions $f:\mathbb{R}\rightarrow\mathbb{R}$
satisfying the additional equation
\begin{equation}\label{Eq1.4}
f(x^{n})=cx^{k}f(x^{m})
\quad
\left(x\in\mathbb{R}\setminus\left\{0\right\}\right),
\end{equation}
where $c\in\mathbb{R}$ and $n, m, k\in\mathbb{Z}$ are arbitrarily
fixed. This approach is obviously the common generalization of the
above mentioned results.
In the second part of the paper we will deal with the stability
of this last system of functional equations.
Our main results could serve as a generalization of the
theorems of \cite{NH68}.
However, the aim of the paper is not only to prove a stability
theorem.
In the so--called mixed theory of information it is usual to consider
a functional equation that characterizes the inset measure of information,
see Maksa \cite{Mak81}.
While solving this equation one obtains an additive function satisfying also the
equation
\[
f\left(\frac{1}{x}\right)=-\frac{1}{x^{2}}f(x)
\quad
\left(x\in\mathbb{R}\setminus\left\{0\right\}\right).
\] 
Clearly, this is a particular case of equation \eqref{Eq1.3}
with $ k = -1 \,$.   
Therefore it is rather natural to expect that 
the investigation of the 
stability of the above mentioned equation 
for the inset measure of information 
should be preceded by the verification of 
the stability of the characterization
of derivations by Kannappan and Kurepa.
Thus our results  
can be applied when we investigate
the stability of a functional equation 
characterizing the inset measure of information.

We remark that in Badora \cite{Bad06}, a stability problem for the
system \eqref{Eq1.1}-\eqref{Eq1.2} concerning mappings between
Banach algebras was solved.
In this paper we replace the second equation \eqref{Eq1.2} 
with an equation in a single variable, namely, with 
a member of the family of equations in the form \eqref{Eq1.4}. 
On the other hand, we restrict our considerations to real functions.  

In what follows we will list some
preliminary definitions and
statements that will be used
during the proof of our main result.
These can be found e.g., in Kuczma \cite{Kuc85}.

Let $p\in\mathbb{N}$ .
A function $f:\mathbb{R}^{p}\rightarrow\mathbb{R}$ is
called \emph{$p$--additive} if, for every
$ i \in \{\, 1 \,,\, 2 \,,\, \dots \,,\, p \,\} $
and for every
$x_{1}, \ldots, x_{p}, y_{i}\in\mathbb{R}$
\begin{multline*}
f\left(x_{1}, \ldots, x_{i-1}, x_{i} + y_{i}, x_{i+1}, \ldots, x_{p}\right)
\\
=f\left(x_{1}, \ldots, x_{i-1}, x_{i}, x_{i+1}, \ldots, x_{p}\right)
+f\left(x_{1}, \ldots, x_{i-1}, y_{i}, x_{i+1}, \ldots, x_{p}\right),
\end{multline*}
i.e., $f$ is additive in each of its variables
$x_{i}\in\mathbb{R}$, $i=1, \ldots, p$.
A $2$--additive function is called \emph{biadditive.}

\begin{thm}\label{T1.1}
Let $f:\mathbb{R}^{p}\rightarrow\mathbb{R}$
be a continuous $p$--additive function.
Then there exists a constant $ c \in \mathbb{R} $ 
such that
\[
f(x_{1} \,,\, x_{2} \,,\, \ldots \,,\, x_{p})=
c x_{1} x_{2} \cdots x_{p} 
\] 
for all $ x_{1} \,,\, x_{2} \,,\, \ldots \,,\, x_{p} \in \mathbb{R} $\,.
\end{thm}

\begin{thm}\label{T1.2}
Let $f:\mathbb{R}^{p}\rightarrow\mathbb{R}$ be a $p$--additive
function,
bounded above, or below on a set $T\subset \mathbb{R}^{p}$,
which has positive Lebesgue--measure.
Then $f$ is continuous.
\end{thm}

Given a function $F:\mathbb{R}^{p}\rightarrow\mathbb{R}$,
by the \emph{diagonalization (or trace)} of $F$ we understand the
function
$f:\mathbb{R}\rightarrow\mathbb{R}$ arising from
$F$ by putting all the variables (from $\mathbb{R}$)
equal:
\[
f(x)=F(x, \ldots, x).
\quad
\left(x\in\mathbb{R}\right)
\]

We will also refer to the definition of \emph{the difference operator}
$\Delta_{h}$ with the span $h\in\mathbb{R}$, which is given for a
function $f:\mathbb{R}\rightarrow\mathbb{R}$ by the formula
\[
\Delta_{h}f(x)=f(x+h)-f(x)
\quad
\left(x\in\mathbb{R}\right).
\]
The superposition of several difference operators will be denoted shortly
by
\[
\Delta_{h_{1}h_{2}\ldots h_{p}}f=\Delta_{h_{1}}\Delta_{h_{2}}\ldots \Delta_{h_{p}}f,
\]
where $p\in\mathbb{N}$ and $h_{1}, h_{2}, \ldots, h_{p}\in\mathbb{R}$.

\begin{lem}\label{L1.1}
Let $F:\mathbb{R}^{p}\rightarrow\mathbb{R}$ be a symmetric
$p$--additive function, and let $f:\mathbb{R}\rightarrow\mathbb{R}$
be the diagonalization of $F$.
For every $n\in\mathbb{N}$, $n\geq p$ and for every
$x, h_{1}, \ldots, h_{n}\in\mathbb{R}$ we have
\[
\Delta_{h_{1}\ldots h_{n}}f(x)=\left\{
\begin{array}{lcl}
p!F\left(h_{1}, \ldots, h_{p}\right), &\text{if} & n=p \\
0, &\text{if} & n\geq p.
\end{array}
\right.
\]
\end{lem}

We remark that according to Theorem 15.1.1 in Kuczma \cite{Kuc85}, we
have for all $p\in\mathbb{N}$,
\[
\Delta_{h_{1}\ldots h_{p}}f(x)=
\sum_{\varepsilon_{1}, \ldots, \varepsilon_{p}=0}^{1}
(-1)^{p-(\varepsilon_{1}+\ldots +\varepsilon_{p})}
f\left(x+\varepsilon_{1}h_{1}+\ldots+\varepsilon_{p}h_{p}\right).
\]

We will also make use of a result of Kannappan--Kurepa \cite{KK70}.
\begin{thm}\label{T1.3}
Let $f, g:\mathbb{R}\rightarrow\mathbb{R}$ be additive functions and
$n, m\in\mathbb{Z}\setminus\left\{0\right\}$, $n\neq m$.
Suppose that
\[
f(x^{n})=x^{n-m}g(x^{m})
\]
holds for all $x\in\mathbb{R}\setminus\left\{0\right\}$.
Then the functions $F, G:\mathbb{R}\rightarrow\mathbb{R}$ defined by
\[
F(x)=f(x)-f(1)x
\quad
\text{and}
\quad
G(x)=g(x)-g(1)x
\quad
\left(x\in\mathbb{R}\right)
\]
are derivations and $nF(x)=mG(x)$ is fulfilled for all $x\in\mathbb{R}$.
\end{thm}

\section{Inequalities for additive functions}

\begin{lem}\label{L2.1}
Let $f, g:\mathbb{R}\rightarrow\mathbb{R}$ be additive functions,
$n, m\in\mathbb{Z}\setminus\left\{0\right\}$, $n\neq m$ suppose furthermore that
either $n=-m$ or $\mathrm{sign}(n)=\mathrm{sign}(m)$
and assume that there exists an interval $I\subset \mathbb{R}$ with
positive length such that
\begin{equation}\label{Eq2.1}
\left|f(x^{n})-x^{n-m}g(x^{m})\right|\leq K
\end{equation}
holds for all $x\in\mathbb{R}\setminus\left\{0\right\}$ with a certain $K\in\mathbb{R}$.
Then there exist derivations $F, G:\mathbb{R}\rightarrow\mathbb{R}$ such that
$nF(x)=mG(x)$  $\left(x\in\mathbb{R}\right)$ and
\begin{equation}\label{Eq2.2}
f(x)=F(x)+f(1)x
\quad
\left(x\in\mathbb{R}\right),
\end{equation}
\begin{equation}\label{Eq2.3}
g(x)=G(x)+g(1)x
\quad
\left(x\in\mathbb{R}\right).
\end{equation}
\end{lem}
\begin{proof}
Firstly, we will show that inequality \eqref{Eq2.1} implies that
there exists $L\in\mathbb{R}$ so that
\begin{equation}\label{Eq2.4}
\left|f(x^{n})-x^{n-m}g(x^{m})\right|\leq L \left|x\right|^{n}
\end{equation}
is fulfilled for all $x\in\mathbb{R}\setminus\left\{0\right\}$.

Let $]a, b[\subset I$. Since the rationals are dense in $\mathbb{R}$, for
every $x\in\mathbb{R}\setminus\left\{0\right\}$ we can find
$r(x)\in\mathbb{Q}$ (a rational number depending only on $x$) such that
$a<r(x) x<b$.
If we replace $x$ by $r(x)x$ in \eqref{Eq2.1}, we obtain that
\[
\left|f(x^{n})-x^{n-m}g(x^{m})\right|\leq K \left|r(x)\right|^{-n},
\]
where we used the fact that every additive function is
$\mathbb{Q}$--homogeneous.
Since $a<r(x)x<b$,
\[
\min\left\{|a|, |b|\right\}< r(x)x < \max\left\{|a|, |b|\right\}.
\]
In case $n>0$, we obtain from this that
\[
\left(\min\left\{|a|, |b|\right\}\right)^{-n} |x|^{n}> |r(x)|^{-n}
\]
and in case $n<0$ we get that
\[
\left(\max\left\{|a|, |b|\right\}\right)^{-n} |x|^{n}> |r(x)|^{-n}.
\]
Therefore, if we define
\[
L=
\left\{
\begin{array}{lcl}
\left(\min\left\{|a|, |b|\right\}\right)^{-n} K, & \text{if} & n>0 \\
\left(\max\left\{|a|, |b|\right\}\right)^{-n} K, & \text{if} & n<0
\end{array}
\right.
\]
we get inequality \eqref{Eq2.4}.

At this point of the proof we have to distinguish several cases.
First suppose that $n, m>0$. Without the loss of generality $n>m$
can be assumed.

Define the function $H$ on $\mathbb{R}^{n}$ by
\begin{multline*}
H(x_{1}, \ldots, x_{n})=
f(x_{1}\cdot\ldots\cdot x_{n}) \\-
\frac{1}{n}x_{1}\cdot\ldots\cdot x_{n-m} g\left(x_{n-m+1}\cdot\ldots\cdot x_{n}\right)-
\frac{1}{n}x_{2}\cdot\ldots\cdot x_{n-m+1} g\left(x_{n-m+2}, \cdot\ldots \cdot x_{n}x_{1}\right)- \\
\ldots -
\frac{1}{n}x_{n}x_{1}\cdot\ldots\cdot x_{n-m-1} g\left(x_{n-m}\cdot\ldots\cdot x_{2}\right).
\end{multline*}
Due to the additivity of the functions $f$ and $g$, the function
$H$ is a symmetric and $n$--additive function, and its trace
\[
H(x, \ldots, x)= f(x^{n})-x^{n-m}g(x^{m}).
\quad
\left(x\in\mathbb{R}\right)
\]
In view of inequality \eqref{Eq2.1}, this yields that
\[
\left|H(x, \ldots, x)\right|=
\left|f(x^{n})-x^{n-m}g(x^{m})\right| \leq
L |x^{n}|,
\quad
\left(x\in\mathbb{R}\setminus\left\{0\right\}\right)
\]
that is, the trace of the function $H$ can be dominated by the term
$L|x^{n}|$.
On the other hand, Lemma \ref{L1.1}.states that the function $H$
is uniquely determined by its trace via the formula
\[
H(h_{1}, \ldots, h_{n})=\frac{1}{n!}\Delta_{h_{1} \ldots h_{n}} H(x, \ldots, x).
\left(x, h_{1}, \ldots, h_{n} \in\mathbb{R}\right)
\]
This yields that the function
$H$ is bounded on a subset of $\mathbb{R}^{n}$
which has positive Lebesgue--measure.
Thus, by Theorem \ref{T1.2}., the function $H$ is continuous on $\mathbb{R}^{n}$.
Therefore, especially,
\[
H(x, \ldots, x)=c x^{n}
\]
holds for all $x\in\mathbb{R}$ with a certain $c\in\mathbb{R}$.
>From this we get that $H(1, \ldots, 1)=c$, on the other hand
by the definition of the function $H$, $H(1, \ldots, 1)= f(1)-g(1)$
follows.
All in all,
\begin{equation}\label{Eq2.6}
\left(f(1)-g(1)\right)x^{n}=H(x, \ldots, x)=
f(x^{n})-x^{n-m}g(x^{m}).
\quad
\left(x\in\mathbb{R}\setminus\left\{0\right\}\right)
\end{equation}
Define the functions $F, G:\mathbb{R}\rightarrow\mathbb{R}$ by
\[
F(x)=f(x)-f(1)x
\quad
\text{and}
\quad
G(x)=g(x)-g(1)x,
\quad
\left(x\in\mathbb{R}\setminus\left\{0\right\}\right)
\]
then from \eqref{Eq2.6} we get that
\[
F(x^{n})=x^{n-m}G(x^{m})
\]
for all $x\in\mathbb{R}\setminus\left\{0\right\}$.
Using Theorem \ref{T1.3}.,
this yields that the functions $F$ and
$G$ are derivations
and $nF(x)=mG(x)$ holds for all $x\in\mathbb{R}$.
This means that equations \eqref{Eq2.2} and \eqref{Eq2.3}
hold in case $n, m>0$.

Secondly assume that $n, m <0$. In this case let us
replace $x$ by $\frac{1}{x}$ in inequality \eqref{Eq2.4}
to obtain
\[
\left|f(x^{-n})-x^{(-n)-(-m)}g(x^{-m})\right|\leq L |x|^{-n}.
\quad
\left(x\in\mathbb{R}\setminus\left\{0\right\}\right)
\]
Since $-n$ and $-m$ are positive integers, the results
of the previous case can be applied.
Therefore there exist derivations $F, G:\mathbb{R}\rightarrow\mathbb{R}$
so that $nF(x)=mG(x)$ $(x\in\mathbb{R})$ and
\[
f(x)=F(x)+f(1)x
\]
and
\[
g(x)=G(x)+g(1)x
\]
holds for all $x\in\mathbb{R}$.

Suppose now that $n=-m$. Then inequality \eqref{Eq2.1} yields that
\begin{equation}\label{Eq2.7}
\left|f(x^{-m})-x^{-2m}g(x^{m})\right|\leq L |x^{-m}|
\end{equation}
holds for all $x\in\mathbb{R}\setminus\left\{0\right\}$.
If we replace $x$ by $x^{1/m}$ $(x>0)$ then inequality \eqref{Eq2.7}
yields that
\begin{equation}\label{Eq2.8}
\left|f\left(\frac{1}{x}\right)-\frac{1}{x^{2}}g(x)\right| \leq \frac{L}{|x|}
\end{equation}
is fulfilled for all $x>0$. Replace in this inequality $x$ by
$x(x+1)$, then
\[
\left|f\left(\frac{1}{x(x+1)}\right)-\frac{1}{x^{2}(x+1)^{2}}g(x(x+1))\right| \leq \frac{L}{|x(x+1)|}
\]
is fulfilled for all $x>0$. After using the additivity of the function f,
\[
\left|f\left(\frac{1}{x}\right)-f\left(\frac{1}{x+1}\right)-\frac{1}{x^{2}(x+1)^{2}}g(x(x+1))\right| \leq \frac{L}{|x(x+1)|}
\quad
\left(x>0\right)
\]
Using the triangle inequality, we obtain that
\begin{multline*}
\left|\frac{1}{x^{2}}g(x)-\frac{1}{(x+1)^{2}}g(x+1)-\frac{1}{x^{2}(x+1)^{2}}g(x^{2}+x)\right|
\\
\leq
\left|f\left(\frac{1}{x}\right)-\frac{1}{x^{2}}g(x)\right|+
\left|f\left(\frac{1}{x+1}\right)-\frac{1}{(x+1)^{2}}g(x+1)\right|
\\
+
\left|f\left(\frac{1}{x(x+1)}\right)-\frac{1}{x^{2}(x+1)^{2}}g(x(x+1))\right|
\\
\leq \frac{L}{|x|}+\frac{L}{|x+1|}+\frac{L}{|x(x+1)|}
\end{multline*}
is satisfied for all $x>0$.
Due to the additivity of the function $g$, after rearranging this inequality,
one can get
\[
\left|2xg(x)-g(x^{2})\right|\leq
L|x|+L|x+1|+L|x(x+1)|+|x^{2}g(1)|
\]
for all $x>0$. At this point the results of the first part of the proof
can be used to derive that there exists a derivation $G:\mathbb{R}\rightarrow\mathbb{R}$ such that
\[
g(x)=G(x)+g(1)x.
\quad
\left(x\in\mathbb{R}\right)
\]
In view of inequality \eqref{Eq2.6}, this yields that
\[
\left|f\left(\frac{1}{x}\right)-\frac{1}{x^{2}}G(x)-\frac{1}{x^{2}}g(1)x\right|\leq \frac{L}{|x|}
\]
is fulfilled for all $x\in\mathbb{R}\setminus\left\{0\right\}$.
Since the function $G$ is a derivation, $-\frac{1}{x^{2}}G(x)=G\left(\frac{1}{x}\right)$
holds for all $x\in\mathbb{R}\setminus\left\{0\right\}$, thus
\[
\left|f\left(\frac{1}{x}\right)+G\left(\frac{1}{x}\right)-\frac{g(1)}{x}\right|\leq \frac{L}{|x|},
\quad
\left(x\in\mathbb{R}\setminus\left\{0\right\}\right)
\]
or if we replace $x$ by $\frac{1}{x}$,
\[
\left|f(x)+G(x)-g(1)x\right|\leq L|x|.
\quad
\left(x\in\mathbb{R}\right)
\]
The functions $f$ and $G$ are additive, therefore the
function $f(x)+G(x)-g(1)x$ is also additive, and this inequality means that
this additive function is bounded on an interval which has positive length. Thus this
function is linear, that is
\[
f(x)+G(x)-g(1)x=cx
\]
holds for all $x\in\mathbb{R}$ with a certain constant $c$. With the
substitution $x=1$, we get however that $c=f(1)-g(1)$. Therefore
\[
f(x)=-G(x)+f(1)x
\]
for all $x\in\mathbb{R}$, that is, equations \eqref{Eq2.2} and \eqref{Eq2.3}
are satisfied in this case, too.
\end{proof}

\begin{rem}
Our proof was not appropriate in case $\mathrm{sign}(n)\neq \mathrm{sign}(m)$  and
$n\neq -m$. We remark that if the functions in the previous lemma fulfill the
condition $\kappa f(x)=g(x)$ for all $x\in\mathbb{R}$ with a real constant $\kappa$, then we
do not have to make any restrictions on the values of $n$ and $m$ and
we can prove the following statement.
\end{rem}

\begin{lem}\label{L2.2}
Let $f:\mathbb{R}\rightarrow\mathbb{R}$ be an additive function, $\kappa\in\mathbb{R}$
$n, m\in\mathbb{Z}\setminus\left\{0, 1\right\}$, $n\neq m$ and assume that
there exists an interval $I\subset \mathbb{R}$ such that
\begin{equation}\label{Eq2.9}
\left|f(x^{n})-\kappa x^{n-m}f(x^{m})\right|\leq K
\end{equation}
holds for all $x\in\mathbb{R}\setminus\left\{0\right\}$
with a certain $K\in\mathbb{R}$.
Then there exists a derivation $F:\mathbb{R}\rightarrow\mathbb{R}$
for which $(n-\kappa m)F(x)=0$  and
\[
f(x)=F(x)+f(1)x
\]
holds for all $x\in\mathbb{R}$.
\end{lem}

\begin{proof}
In view of Lemma \ref{L2.1}. we have to only deal with the
case $\mathrm{sign}(n)\neq\mathrm{sign}(m)$ and $n\neq -m$.
Furthermore, due the proof the previous lemma, inequality
\eqref{Eq2.9} is equivalent to the following inequality
\[
\left|f(x^{n})-\kappa x^{n-m}f(x^{m})\right|\leq L |x^{n}|
\]
for all $x\in\mathbb{R}\setminus\left\{0\right\}$, where
$L$ is a certain real constant.
Let us substitute $x^{n}$ in place of $x$ into this
inequality,
\[
\left|f(x^{n^{2}})-\kappa x^{n(n-m)}f(x^{nm})\right|\leq L |x^{n^{2}}|.
\quad
\left(x\in\mathbb{R}\setminus\left\{0\right\}\right)
\]
Additionally, the above inequality with the substitution
$x^{m}$ yields
\[
\left|f(x^{n m})-\kappa x^{m(n-m)}f(x^{m^{2}})\right|\leq L |x^{nm}|.
\quad
\left(x\in\mathbb{R}\setminus\left\{0\right\}\right)
\]
This last two inequalities and the triangle inequality
imply that
\begin{multline*}
\left|f\left(x^{n^{2}}\right)-\kappa^{2} x^{n^{2}-m^{2}}f\left(x^{m^{2}}\right)\right| \\
\leq
\left|f(x^{n^{2}})-\kappa x^{n(n-m)}f(x^{nm})\right|+
\left|\kappa x^{n(n-m)}\right| \cdot
\left|f(x^{n m})-\kappa x^{m(n-m)}f(x^{m^{2}})\right| \\
=L\left|x^{n^{2}}\right|+L\left|\kappa x^{n(n-m)}x^{nm}\right|=
L\left(1+\left|\kappa\right|\right) \cdot \left|x^{n^{2}}\right|
\end{multline*}
holds for all $x\in\mathbb{R}\setminus\left\{0\right\}$ .
Let us observe that $n^{2}, m^{2}>0$ and the results of
Lemma \ref{L2.1}. can be applied
(with the choice $g(x)=\kappa f(x)$) to obtain that
\[
f(x)=F(x)+f(1)x
\]
holds for all $x\in\mathbb{R}$, where $F:\mathbb{R}\rightarrow\mathbb{R}$
is a derivation which also satisfies $(n-\kappa m)F(x)=0$
for arbitrary $x\in\mathbb{R}$.
\end{proof}

>From this lemma the following statement can be concluded immediately.

\begin{cor}
Let $f:\mathbb{R}\rightarrow\mathbb{R}$ be an additive function and
$r\in\mathbb{Q}\setminus\left\{0, 1\right\}$. Assume that there exists an
interval $I\subset \mathbb{R}$ with positive length such that
\[
\left|f\left(x^{r}\right)-rx^{r-1}f(x)\right|\leq K
\]
holds for all $x\in I$ with a certain $K\in\mathbb{R}$.
Then there exists a derivation $F:\mathbb{R}\rightarrow\mathbb{R}$
such that
\[
f(x)=F(x)+f(1)x
\]
is fulfilled for all $x\in\mathbb{R}$.
\end{cor}

\section{Stability of derivations and linear functions}

As a starting point of the proof of the main result of this section
the theorem of Hyers will be used. Originally this
statement was formulated in terms of functions that
are acting between Banach spaces, see Hyers \cite{Hye41}.
However, we will use this theorem only in the particular case 
when the domain and 
the range are the set of reals.
In this setting the proposition 
of Hyers' theorem is the following.

\begin{thm}\label{T3.1}
Let $\varepsilon\geq 0$ and suppose that the function
$f:\mathbb{R}\rightarrow\mathbb{R}$ fulfills the inequality
\[
\left|f(x+y)-f(x)-f(y)\right|\leq \varepsilon
\]
for all $x\in\mathbb{R}$. Then there exists an additive
function $a:\mathbb{R}\rightarrow\mathbb{R}$ such that
\[
\left|f(x)-a(x)\right|\leq \varepsilon
\]
holds for arbitrary $x\in\mathbb{R}$.
\end{thm}

Applying this theorem and our results in the previous section, 
we can establish our main result. 

\begin{thm}\label{T3.2}
Let $\varepsilon_{1}, \varepsilon_{2}\geq 0$, $\kappa\in\mathbb{R}$,
$n, m\in\mathbb{Z}\setminus\left\{0\right\}$, $n\neq m$ and
assume that the function $f:\mathbb{R}\rightarrow\mathbb{R}$
fulfills the 
inequalities
\begin{equation}\label{Eq3.1}
\left|f(x+y)-f(x)-f(y)\right|\leq \varepsilon_{1}
\end{equation}
\and
\begin{equation}\label{Eq3.2}
\left|f(x^{n})-\kappa x^{n-m}f(x^{m})\right|\leq \varepsilon_{2}
\end{equation}
for all $x\in\mathbb{R}\setminus\left\{0\right\}$.
Then there exist a derivation $F:\mathbb{R}\rightarrow\mathbb{R}$
and $\lambda\in\mathbb{R}$ such that
\[
(n - \kappa m) F(x) = 0 
\]
and
\begin{equation}\label{Eq3.3}
\left|f(x)-\left[F(x)+\lambda x\right]\right|\leq \varepsilon_{1}
\end{equation}
are satisfied for all $x\in\mathbb{R}$.
\end{thm}

\begin{proof}
Due the theorem of Hyers, inequality \eqref{Eq3.1}
immediately implies that there exists an additive function
$a:\mathbb{R}\rightarrow\mathbb{R}$ satisfying 
\begin{equation}\label{Eq3.4}
\left|f(x)-a(x)\right|\leq \varepsilon_{1}
\end{equation}
for all $x\in\mathbb{R}$.
In view of inequality \eqref{Eq3.2} this implies that
\begin{multline*}
\left|a(x^{n})-\kappa x^{n-m}a(x^{m})\right| \\
\leq
\left|a(x^{n})-f(x^{n})\right|+
\left|\kappa x^{n-m}\right|\cdot \left|a(x^{m})-f(x^{m})\right|+
\left|f(x^{n})-\kappa x^{n-m}f(x^{m})\right|
\\
\leq \varepsilon_{1}+\left|\kappa x^{n-m}\right|\varepsilon_{1}+\varepsilon_{2}=
\left(1+\left|\kappa x^{n-m}\right|\right)\varepsilon_{1}+\varepsilon_{2}
\end{multline*}
is fulfilled for all $x\in\mathbb{R}\setminus\left\{0\right\} \,$.  
Thus the expression $\left|a(x^{n})-\kappa x^{n-m}a(x^{m})\right|$ is
bounded on a real interval with non-void interior.
Therefore Lemma \ref{L2.2}. yields that there exists a derivation
$F:\mathbb{R}\rightarrow\mathbb{R}$ 
such that $ (n-\kappa m) F(x) = 0 $ and
\[
a(x)=F(x)+a(1) x
\]
holds for all $x\in\mathbb{R}$.
This, together with \eqref{Eq3.4}, implies \eqref{Eq3.3}
with $ \lambda = a(1) \,$.
\end{proof}

Let us note that our result serves as a stability theorem for 
linear functions if $ n \neq \kappa m \,$.

\end{document}